\documentclass[12pt]{article}
\author{Anirban Das}
\title{Modeling Avalanches of Neurons}

\usepackage{makeidx} 
\usepackage[titletoc,title]{appendix}
\usepackage{bbm}
\usepackage{amssymb}
\usepackage{amsmath,amsfonts,amssymb,amsthm,epsfig,epstopdf,titling,url,array}
\makeindex
\usepackage[left=2cm,right=2cm,top=3cm,bottom=4cm]{geometry} 
\theoremstyle{plain}
\newtheorem{thm}{Theorem}[section]
\newtheorem{lem}[thm]{Lemma}

\theoremstyle{definition}
\newtheorem{defn}{Definition}[section]

\theoremstyle{remark}

\begin{document}

\date{}
\title
{An analytic derivation of the variance for the Abelian distribution}
\author{Anirban Das, \\
Department of Mathematics, Pennsylvania State University\\
anirban.das.psu.math@gmail.com
}

\maketitle

\begin{abstract}{ The Abelian distribution has been studied recently in models for neural avalanches. This paper uncovers new properties about the distribution, ways in which these properties can be useful are indicated.}
\end{abstract}


\section{Introduction}
The Abelian distribution is a distribution that is in important in models studying neural avalanches (See \cite{eurich2002finite}, \cite{levina2008mathematical} \cite{levina2014abelian}). Neural avalanches were observed  by John Beggs and
Dietmar Plenz (\cite{beggs2003neuronal}, \cite{beggs2004neuronal}). In the experiment cultured slices  were planted on a multielectrode array and local field potential signals were recorded. The data consisted of short intervals of activity, when one or more electrodes detected LFPs above the threshold, separated by longer periods of silence. A set of such consecutively
active frames was called an avalanche. The size of an avalanche is defined as
the number of electrodes which were active during the avalanche. The data collected showed that avalanche sizes followed the power-law distribution (with exponent $-\frac{3}{2}$) with the exponential cutoff at the
size of the multielectrode array. Neuronal avalanches have also been recently identified in vivo in the normalized LFPs extracted from ongoing activity in awake macaque monkeys(\cite{petermann2009spontaneous}). A model for such neural avalanches was studied in 
\cite{eurich2002finite}, in the model the probability distribution of getting avalanches of  size $Z_{N,p}$, was derived as $$P(Z_{N,p} = b)= C_{N,p}{{N-1}\choose{b-1}}p^{b-1}(1-bp)^{(N-b-1)}b^{b-2}$$, where $C_{N,p}$ is the normalization constant defined by 
$C_{N,p}= \frac{1-Np}{1-(N-1)p} .$ This is the Abelian distribution. The  parameter $N$ is the number of neurons, the parameter $p= \frac{\alpha}{N}$, where $\alpha \in (0,1)$ captures the amount of dissipation in the system. 

In \cite{denker2014ergodicity} a model for studying avalanches in dynamical systems was constructed, and a similar distribution was derived, we shall call it the Avalanche distribution. In \cite{levina2008mathematical}, the Mean of the Abelian distribution was calculated. This paper calculates the variance of the Abelian distribution analytically. We shall use properties of the Avalanche distribution, for doing this. This underlines the close relationship between the two distributions. At the end of \cite{denker2014ergodicity}, there is a note by Wenbo Li(\cite{WenboVLi}), where Li states some results (without proof)  and claims to have found a way for calculating the variance(does not give an explicit formula) of the Abelian distribution. Unfortunately Dr Li died before the results could be published. Here we prove some of Dr Li's claims, and use them to calculate the Variance of the Abelian distribution.  Also we show what happens to the variance as the parameter $N$(\ref{def_ab_dist}) goes to $ + \infty$.

The first Section, deals with the calculation of the Variance for a given $N$ and  $p$. the second section introduces the Stirling numbers.  Some properties of the Stirling numbers are derived. These will be used in Section 3, to study how the variance of the distribution behaves as the parameter $N$ goes to infinity. Finally he Appendix contains some technical proofs that have been excluded from the main text.

\section{The variance of the Abelian distribution}
\begin{defn}{ \textbf{ Abelian Distribution }}{\label{def_ab_dist}}
\\ The Abelian distribution $Z_{N,p}$ is a probability distribution on $\{1,2, \cdots ,N\}$defined by the probability density
$$P(Z_{N,p} = b)= C_{N,p}{{N-1}\choose{b-1}}p^{b-1}(1-bp)^{(N-b-1)}b^{b-2}$$ Where $C_{N,p}$ is the normalization constant defined by 
$C_{N,p}= \frac{1-Np}{1-(N-1)p} .$ The parameter $N$ must be an integer, the parameter $p$ lies in $(0,\frac{1}{N})$.
\end{defn}
That this is indeed is a distribution was proved in \cite{levina2008mathematical}, see also \cite{levina2014abelian}. The $p$ in the Abelian distribution is often taken as $\frac{\alpha}{N}$, where $0<\alpha < 1$. It was also proved in \cite{levina2008mathematical}, \cite{levina2014abelian} that :
\begin{lem}\label{ab_mean}
$ E(Z_{N,\frac{\alpha}{N}})= \frac{N}{N-(N-1)\alpha}$.
\end{lem}

In \cite{denker2014ergodicity} we find the following distribution:
\begin{defn}{ \textbf{ Avalanche Distribution }}
The Avalanche distribution $X_{N,p}$ is a probability distribution on $\{0,1,2, \cdots ,N\}$ defined by the probability density
$$P(X_{N,p} = b)= {{N}\choose{b}}p^{b}(1-(b+1)p)^{N-b}(b+1)^{b-1}.$$The parameter $N$ must be an integer, the parameter $p$ lies in $(0,\frac{1}{N})$.
\end{defn}
Wenbo Li in \cite{WenboVLi} states without proof the following result about the mean of the Avalanche distribution. We attach a proof of the statement in the Appendix.
\begin{lem}{\label{av_mean}}
 $ E(X_{N,p})= \sum\limits_{i=1}^{N} \frac{N !}{(N-i)!} p^i.$
\end{lem}

Also define $Y_{N,p} \; = \; X_{N,p}+1$.
Thus 
\begin{eqnarray*}
P(Y_{N,p} = b) &=& P(X_{N,p} = b-1)\\
& =& {{N}\choose{b-1}}p^{b-1}(1-bp)^{(N-b+1)}b^{b-2}.
\end{eqnarray*}

With these results at hand we are ready to find the Variance of $Z_{N,p}$
\begin{thm}
The second moment of the Abelian distribution is as follows:
\begin{equation}\label{eq2_2}
   E({Z_{N,p}}^{2})= \frac{{C_{N,p}}}{p}[\frac{1}{1-Np}-1-\sum\limits_{i=1}^{N-1} {\frac{(N-1)!}{(N-1-i)!}p^i}]
\end{equation}
And the variance of the distribution is
\begin{equation}
V(Z_{N,p})=\frac{{C_{N,p}}}{p}[\frac{1}{1-Np}-1-\sum\limits_{i=1}^{N-1} {\frac{(N-1)!}{(N-1-i)!}p^i}] - {(\frac{N}{N-(N-1)\alpha})}^{2}.
\end{equation}

 \end{thm}
 \begin{proof}
 \begin{eqnarray*}
 E(Y_{N,p}) &=& \sum\limits_{b=1}^{N+1} b^{b-1} {{N}\choose{b-1}}p^{b-1}(1-bp)^{(N-b+1)} \\
 &=& \sum\limits_{b=1}^{N+1} b^{b-1} {{N}\choose{b-1}}p^{b-1}(1-bp)^{(N-b)} - \sum\limits_{b=1}^{N+1}  p \times b^{b} {{N}\choose{b-1}}p^{b-1}(1-bp)^{(N-b)} \\
 &=& \frac{1}{C_{N+1,p}} E(Z_{N+1,p})- p \frac{1}{{C_{N+1,p}}}E({Z_{N+1,p}}^2)
 \end{eqnarray*}
 We know $E(Z_{N+1,p})$ from \ref{ab_mean}, 
  using \ref{av_mean} we can compute $E(Y_{N,p})$.
 Using the above two facts one finds $ E({Z_{N,p}}^{2})$. Since $V(Z_{N,p})= E({Z_{N,p}}^{2}) - {E(Z_{N,p})}^{2} $ , the variance too can be found from this.\\
 \end{proof}

 \section{Stirling numbers}
 Our chief goal , for the rest of this paper will be to find how the variance of the Abelian distribution behaves as $N$ goes to infinity. In order to do this we shall use the Stirling number of the first kind. The Stirling numbers were so named by N. Nielson (1906) in honor of James Stirling, who introduced them in his Methodus Differentialis (1730) \cite{tweddle2012james}, without using any notation for them. The notation in this paper is due to J. Riordan \cite{riordan2012introduction}. This section gives some definitions, and results from  \cite{charalambides2005combinatorial}. We then proceed to state and prove a few Lemmas of our own \ref{lemma_J_5_1}, \ref{lemma_J_5_2}, \ref{lemma_J_6}. These final three will be used in the section titled Asymptotic behavior of the Variance of the Abelian Distribution.\\

We will use the notation$(x)_{n}$ for the polynomial $x(x-1)(x-2) \cdots (x-(n-1))$. This is called the factorial moment of order $n$.

The coefficients of such polynomials are called the Stirling numbers. Formally we have,
$(x-r)_{i}= \sum\limits_{j=0}^{i} s(i,j;r) x^{j} $. Set $s(0,0;r)=1$. For $i\ge j \ge 0$, $s(i,j;r)$ are called the non- centered Stirling numbers of the first kind.
We will be chiefly interested in r= 1, when
 $(x-1)_{i}= \sum\limits_{j=0}^{i} s(i,j;1) x^{j}$. When $i \ge j >0$ denote by ${\tau^{i}}_{j}$ the class of all possible subsets of  $\{1,2,3 \cdots i  \}$, which are of cardinality $j$. 
The following can be found in Chapter 2 of \cite{charalambides2005combinatorial}.
\begin{equation}\label{eq2_4}
  | s(i,j;1)| = (-1)^{i-j}  s(i,j;1)
\end{equation}
\begin{equation}\label{eq2_5}
 (x+i)_{i}= \sum\limits_{j=0}^{i} |s(i,j;1)| x^{j} 
\end{equation}
\begin{equation}\label{eq2_7}
 s(i,i;1)=1
\end{equation}
\begin{equation}\label{eq2_8}
 s(i,i-1;1)=-\frac{i(i+1)}{2}
\end{equation}

\begin{equation}\label{eq2_6}
|s(i,j;1)|= i! \sum _{\{r_1,r_2, \cdots , r_j \} \in {\tau^{i}}_{j}} \frac{1}{r_1 r_2 \cdots r_j} \text{, This holds for }i \ge j >0 
\end{equation}

Now we make some definitions, and prove some results of interest to us.

\begin{defn} \label{definition_5_4}
Given positive integer $i$. $P_i$ is a polynomial of degree $i (i\ge 0)$ defined as $P_i(x)= \sum\limits_{j=0}^{i}s(i+2,j;1)x^{j}$.
$h_i$ is a polynomial of degree $i+2$ defined as $h_i(x)= x^{i+1}(\frac{(i+2)(i+3)}{2}-x)$.
\end{defn}

\begin{lem}\label{lemma_J_6}
$i, N$ be positive integers, $N-3 \ge i$. Then 
$$
P_i(N)={(N-1)}_{i+2}+h_i(N)
$$
Further when $N-3 \ge i \ge \sqrt{2N}$, $h_i(N) >0$, and also $2N^{i+3}>P_i(N)>{(N-1)}_{i+2} \ge 0 $
\end{lem}
\begin{proof}
The proof is a straightforward calculation.
\end{proof}

\begin{lem}\label{lemma_J_5_1}
There exists a polynomial $f(x)$, of degree $4$,  such that for all integers with  $i,j$ with $i \ge j \ge 0$, we have $f(i) \ge 0$ and
\begin{equation}
  |s(i+2,j;1)| \leq |s(i,j;1)|f(i).
\end{equation}
\end{lem}
\begin{proof}
See Appendix for Proof
\end{proof}

Before ending the section we will state one last technical Lemma that will find use in the next section.
\begin{lem}\label{lemma_J_5_2}
When $0  \le i < \sqrt{2N}$,  $\prod_{j=1}^{i}(1+\frac{j}{N}) \le e^{2}$.
\end{lem}

\begin{proof}
It would be enough to show $\sum_{j=1}^{i}\ln (1+\frac{j}{N}) \le 2$. This is exactly what we do.
\begin{align*}
    \sum_{j=1}^{i}\ln (1+\frac{j}{N}) &\le \sum_{j=1}^{i} \frac{j}{N} && \text{(Using the fact } \ln (1+x) \le x, \text {when} \;x >-1)\\
     &= \frac{1}{2N} \times(i)(i+1) \le 2     
\end{align*}
\end{proof}

\section{Asymptotic behavior of the variance of the Abelian Distribution}
We will see how the variance of $Z_{N,p}$ behaves as $N$ tends to infinity, here we take $p=\frac {\alpha}{N}$. Here is the result.
\begin{thm}
$Z_{N,p}$ be the Abelian distribution with parametres $p$ and $N$. For $ 0 <\alpha < 1$, $\lim_{N \to +\infty} V(Z_{N, \frac{\alpha}{N}})= \frac{\alpha}{(1- \alpha)^{3}} $
\end{thm}
\begin{proof}
Let $p= \frac{\alpha}{N}$. Restating \ref{eq2_2} we get
$$
 E({Z_{N,p}}^{2})= \frac{{C_{N,p}}}{p}[\frac{1}{1-Np}-1-\sum\limits_{i=1}^{N-1} (N-1)_{i}p^i]
$$
The fact that $s(i,i,1)=1$ is used in the following calculations
\begin{eqnarray*}
 E({Z_{N,\frac{\alpha}{N}}}^{2}) &=&  \frac{{C_{N,p}}}{p}[\frac{1}{1-Np}-1-\sum\limits_{i=1}^{N-1} p^i\sum\limits_{j=0}^{i}s(i,j;1)N^{j}]\\
 &=&  \frac{{C_{N,p}}}{p}[\sum\limits_{i=1}^{\infty}(Np)^i-\sum\limits_{i=1}^{N-1} p^i\sum\limits_{j=0}^{i}s(i,j;1)N^{j}] \\
 &=& \frac{{C_{N,p}}}{p}[\sum\limits_{i=N}^{\infty}(Np)^i-\sum\limits_{i=1}^{N-1} p^i\sum\limits_{j=0}^{i-1}s(i,j;1)N^{j}] \\
 &=& \frac{{C_{N,p}}}{p}[\frac{{\alpha}^N}{1-\alpha}-\sum\limits_{i=1}^{N-1} p^i\sum\limits_{j=0}^{i-1}s(i,j;1)N^{j}]
 \end{eqnarray*}
 Hence we have 
 \begin{equation}\label{eq2_10}
  E({Z_{N,\frac{\alpha}{N}}}^{2}) = C_{N,p}[J_1- J_2]
\end{equation}
where $J_1= \frac{{\alpha}^N}{\frac{\alpha}{N}(1-\alpha)}$, $J_2 = \frac{1}{p}\sum\limits_{i=1}^{N-1} p^i\sum\limits_{j=0}^{i-1}s(i,j;1)N^{j} $\\
We easily observe
\begin{eqnarray*}
\lim_{N \to +\infty}C_{N,p}=lim_{N \to +\infty}C_{N,\frac{\alpha}{N}} =1, \;
\lim_{N \to +\infty}J_1= 0.
\end{eqnarray*}
Now to focus on $J_2$
\begin{eqnarray*}
 J_2 &=&  \frac{1}{p}p\sum\limits_{i=0}^{N-2} p^i\sum\limits_{j=0}^{i}s(i+1,j;1)N^{j}\\
 &=& \sum\limits_{i=0}^{N-2} p^i\sum\limits_{j=0}^{i}s(i+1,j;1)N^{j} \\
 &=& \sum\limits_{i=0}^{N-2}p^{i}N^{i}s(i+1,i;1) \; \;+\; \; \sum\limits_{i=1}^{N-2} p^i\sum\limits_{j=0}^{i-1}s(i+1,j;1)N^{j}
 \end{eqnarray*}
 So we have 
 \begin{equation}\label{eq2_11}
  J_2=J_3+ J_4
\end{equation}
where $J_3=\sum\limits_{i=0}^{N-2}p^{i}N^{i}s(i+1,i;1) $, and $J_4=\sum\limits_{i=1}^{N-2} p^i\sum\limits_{j=0}^{i-1}s(i+1,j;1)N^{j} $\\
\begin{eqnarray*}
 J_3 &=& \sum\limits_{i=0}^{N-2} {\alpha}^{i}s(i+1,i;1) 
= - \sum\limits_{i=0}^{N-2} {\alpha}^{i}\frac{(i+1)(i+2)}{2}
 \end{eqnarray*}
This yields
$$
lim_{N \to +\infty}J_3= -\frac{1}{(1-\alpha)^{3}}
$$
\begin{eqnarray*}
J_4 &=& \sum\limits_{i=1}^{N-2} p^i\sum\limits_{j=0}^{i-1}s(i+1,j;1)N^{j}\\
 &=& p \sum\limits_{i=0}^{N-3} p^i\sum\limits_{j=0}^{i}s(i+2,j;1)N^{j}\\
 &=& p \sum\limits_{i=0}^{\sqrt{2N}-1} p^i\sum\limits_{j=0}^{i}s(i+2,j;1)N^{j} +
p \sum\limits_{i=\sqrt{2N}}^{N-3} p^i\sum\limits_{j=0}^{i}s(i+2,j;1)N^{j}\\
&=& J_5+ J_6,
\end{eqnarray*}
where $J_5=\sum\limits_{i=0}^{\sqrt{2N}-1} p^i\sum\limits_{j=0}^{i}s(i+2,j;1)N^{j}$ and $J_6=p \sum\limits_{i=\sqrt{2N}}^{N-3} p^i\sum\limits_{j=0}^{i}s(i+2,j;1)N^{j} $.
\begin{description}
\item[We next show $lim_{N \to +\infty} J_5 =0, f$ below is defined in \ref{lemma_J_5_1}.]
\begin{align*}
|J_5| &= |p(\sum\limits_{i=0}^{\sqrt{2N}-1} p^i\sum\limits_{j=0}^{i}s(i+2,j;1)N^{j})|\\
&\le p(\sum\limits_{i=0}^{\sqrt{2N}-1} p^i\sum\limits_{j=0}^{i}|s(i+2,j;1)|N^{j})\\
&\le p(\sum\limits_{i=0}^{\sqrt{2N}-1} p^i\sum\limits_{j=0}^{i}f(i)|s(i,j;1)|N^{j}) \\
&\le p(\sum\limits_{i=0}^{\sqrt{2N}-1} p^if(i)\prod_{j=1}^{i}(N+j))\\
&\le \frac{\alpha}{N}(\sum\limits_{i=0}^{\sqrt{2N}-1} \alpha^if(i)\prod_{j=1}^{i}(1+\frac{j}{N}))&& \text{(Putting } p =\frac{\alpha}{N})\\
&\le \frac{\alpha}{N}(\sum\limits_{i=0}^{\sqrt{2N}-1} \alpha^if(i)e^{2})
\end{align*}
Since the last expression is an upper bound for $|J_5|$, and it approaches zero as $N$ approaches $\infty$, we are done.

\item[Next we will show $lim_{N \to +\infty} J_6 =0 $, the $P_i$ below is defined in \ref{lemma_J_6}]
\begin{align*}
|J_6| &= |p(\sum\limits_{i=\sqrt{2N}}^{N-3} p^i P_i(N))|\\
&= p(\sum\limits_{i=\sqrt{2N}}^{N-3} p^i P_i(N))\\
&\le p(\sum\limits_{i=\sqrt{2N}}^{N-3} p^i 2 N^{i+3} )\\
&= \frac{\alpha}{N}(\sum\limits_{i=\sqrt{2N}}^{N-3} \alpha^i 2 N^{3} )&& \text{(Putting } p =\frac{\alpha}{N})\\
&= 2N^{2} \alpha (\sum\limits_{i=\sqrt{2N}}^{N-3} \alpha^i )
= 2N^{2} \alpha^{\sqrt{2N}+1} \frac{1-\alpha^{(N-3)-2\sqrt{N}}}{1-\alpha}.
\end{align*}
Since the last expression is an upper bound for $|J_6|$, and it approaches zero as $N$ approaches $\infty$, we are done.
\end{description}

 From (\ref{eq2_10}), and the estimations of $J_1,\; J_3, \; J_4$,
we see 
$$
 lim_{N \to +\infty} E({Z_{N,\frac{\alpha}{N}}}^{2})= (1-\alpha)^{-3}
$$
Also
$$
 lim_{N \to +\infty} {E(Z_{N,\frac{\alpha}{N}})}^{2}=   lim_{N \to +\infty} \frac{1}{1-\frac{(N-1)\alpha}{N}}= \frac{1}{(1-\alpha)^2}
$$
From this the theorem follows.

\end{proof}

\section{Remarks}
The choice of $\alpha$ which is most interesting for studying the biological phenomenon, are those values close to $1$, these are the values for which the distribution behaves closest to a power law as has been show in \cite{eurich2002finite}. Also Levina \cite{levina2007dynamical} shows these are the values of $\alpha$  the system settles to if one starts with dynamical synapses with a different suitable $\alpha$. This result shows that such systems have very high variance when we deal with a lot of neurons. This is consistent with a power law distribution of exponent $-\frac{3}{2}$ as observed in the experiments of \cite{beggs2004neuronal}, \cite{beggs2003neuronal}. An explicit expression for Variance , as has been found here, often proves useful for inferring details about parameters from available data.
\begin{appendices}
\section{Methods}
Here, we give proofs of certain lemmas, that we omitted from the main text.
 \begin{proof}{ of \ref{av_mean}}
 Let $(U_{i})_{i=1}^{N}$ be i.i.d uniformly distributed in $[0,1]$, $N$ is a fixed positive integers,  $p$ is a fixed number in $(0,\frac{1}{N})$  .  From this one may recursively construct the random sequences $(\epsilon_{i,N})_{i=1}^{N}$ as follows

\begin{align*}
   \epsilon_{0,N}= 1 , \;
   \epsilon_{1,N}=\sum\limits_{j=1}^{N} \mathbbm{1}_{[1-p,1]}(U_j) \\
   \epsilon_{k,N}= \sum\limits_{j=1}^{N}\mathbbm{1}_{[1-p\sum\limits_{i=0}^{k-1} \epsilon_{i,N}, \; 1-p\sum\limits_{i=0}^{k-2} \epsilon_{i,N} ]} (U_j),\;
 S_{N,p} = \sum\limits_{i=1}^{N} \epsilon_{i,N}
\end{align*}

 It was shown in \cite{denker2014ergodicity} that $S_{N,p}$ has the same distribution as the Avalanche distribution.
 Thus it  is suffice to prove that
 $ E(\epsilon_{k,N})=  \frac{N !}{(N-k)!} p^k, \; \; \forall k \geq 1$, \\ We do so by induction for $k=1$, 
 
 \begin{eqnarray*}
     E(\epsilon_{1,N}) &=& \sum\limits_{i=1}^{N} P(U_i > 1-p)  
       = \sum\limits_{i=1}^{N} p
      = Np
\end{eqnarray*}
By Inductive hypothesis the result holds for $k=k-1$, now for $k=k$,
\begin{eqnarray*}
     E(\epsilon_{k,N}) &=& \sum\limits_{i=1}^{N} P(p\sum\limits_{m=0}^{k-1} \epsilon_{m,N} \ge 1-U_i \ge p\sum\limits_{m=0}^{k-2}\epsilon_{m,N})  \\ 
      & =& \sum\limits_{i=1}^{N} \sum\limits_{j=1}^{N-1} jp \; P(\epsilon_{k-1,N-1}= j)\\
      & =& Np \sum\limits_{j=1}^{N-1} j P(\epsilon_{k-1,N-1}= j)\;\; \; \; \; \; \; \textbf{ [Taking conditions on value of $\epsilon_{k-1,N}$}]\\
      & =& Np \frac{N-1 !}{(N-k )!} p^{k-1} \;\; \; \; \; \; \; \; \; \; \;\; \; \; \; \; \;\; \; \; \;  \textbf{[By Inductive Hypothesis]}\\
      & =& \frac{N!}{(N-k)!} p^{k}.
\end{eqnarray*}

\end{proof}

\begin{proof}{of \ref{lemma_J_5_1}}
For the moment , consider $i \ge j > 0$, the situation where $i \ge j=0$, will be treated at the end separately.
Using Equation  ~\ref{eq2_6}, we get 

\begin{eqnarray*}
|s(i+2,j;1)| & = & (i+2)! \sum _{\{r_1,r_2, \cdots , r_j \} \in {\tau^{i}}_{j}} \frac{1}{r_1 r_2 \cdots r_j} \\
& &+ (i+2)!\sum _{\{r_1,r_2, \cdots , r_{j-1} \} \in {\tau^{i}}_{j-1}} \frac{1}{(i+1)r_1 r_2 \cdots r_{j-1}}  \\
& & +  (i+2)! \sum _{\{r_1,r_2, \cdots , r_{j-1} \} \in {\tau^{i}}_{j-1}} \frac{1}{(i+2)r_1 r_2 \cdots r_{j-1}}\\
& & + (i+2)!\sum _{\{r_1,r_2, \cdots , r_{j-2} \} \in {\tau^{i}}_{j-2}} \frac{1}{(i+1)(i+2)r_1 r_2 \cdots r_{j-2}}\label{initial equation}
\end{eqnarray*}

  Now for $i \ge j >0$ consider the function $F_{i.j}:\tau^{i}_{j-1}\rightarrow \tau^{i}_{j}$ ($F_{i,j}$ is a function which takes sets to sets)defined as 
 $$
 F_{i,j}(\{ r_1,r_2, \cdots , r_{j-1} \})= \{ l, r_1,r_2, \cdots , r_{j-1} \}
 $$
 where $l$ is the least number in $\{ 1,2, \cdots , i \}$, which is not in $\{ r_1,r_2, \cdots , r_{j-1} \}$

\begin{equation*}
  \forall K   \in \ \tau^{i}_{j}, \; \; \; |{F_{i,j}}^{-1}(K)| \le j \le i
\end{equation*}
Also 
\begin{equation*}
  \forall \{ r_1,r_2, \cdots , r_{j-1} \} \in \tau^{i}_{j-1}, \;\; \; \frac{1}{ (i+1)r_1 r_2 \cdots r_{j-1}} \le 
  \frac{1}{\prod_{g \in F_{i,j}(\{ r_1,r_2, \cdots , r_{j-1} \})} g}
\end{equation*}
It follows that 

\begin{eqnarray*}
\sum _{\{r_1,r_2, \cdots , r_{j-1} \} \in {\tau^{i}}_{j-1}} \frac{1}{(i+1)r_1 r_2 \cdots r_{j-1}} 
& \le & \sum _{\{r_1,r_2, \cdots , r_{j-1} \} \in {\tau^{i}}_{j-1}} \frac{1}{\prod_{g \in F_{i,j}(\{ r_1,r_2, \cdots , r_{j-1} \})} g}\\
& \le &  \sum _{\{r_1,r_2, \cdots , r_j \} \in {\tau^{i}}_{j}} |{F_{i,j}}^{-1}(\{r_1,r_2, \cdots , r_j \})| \frac{1}{(r_1 r_2 \cdots r_j)}\\
& \le & i \sum _{\{r_1,r_2, \cdots , r_j \} \in {\tau^{i}}_{j}} \frac{1}{r_1 r_2 \cdots r_j}
\end{eqnarray*}
Thus
$$
\frac{i \times (i+2)!}{(i)!}|s(i,j;1)| \; \; \;  \geq (i+2)!\sum _{\{r_1,r_2, \cdots , r_{j-1} \} \in {\tau^{i}}_{j-1}} \frac{1}{(i+1)r_1 r_2 \cdots r_{j-1}}
$$.
Similarly 
$$
\frac{i \times (i+2)!}{(i)!}|s(i,j;1)|\; \; \geq (i+2)!\sum _{\{r_1,r_2, \cdots , r_{j-1} \} \in {\tau^{i}}_{j-1}} \frac{1}{(i+2)r_1 r_2 \cdots r_{j-1}}
$$
and 
$$
\frac{i \times (i-1) \times (i+2)!}{(i)!}|s(i,j;1)| \; \; \; \geq (i+2)!\sum _{\{r_1,r_2, \cdots , r_{j-2} \} \in {\tau^{i}}_{j-2}} \frac{1}{(i+1)(i+2)r_1 r_2 \cdots r_{j-2}}
$$
Using the above three in the initial equation for $|s(i+2,j;1)|$, we get 
\begin{eqnarray}
|s(i+2,j;1)|  &\le&  ((i+1)(i+2)+2(i+1)(i+2)i + (i+1)(i+2)i(i-1))|s(i,j;1)|\\
&\le &((i+1)(i+2)+2(i+1)(i+2)i + (i+1)(i+2)i(i-1)+4)|s(i,j;1)| \label{appendix_eq_1}
\end{eqnarray}
The polynomial $((x+1)(x+2)+2(x+1)(x+2)x + (x+1)(x+2)x(x-1))+4$ is defined as $f$, we have shown above that it satisfies the prescribed properties for $i \ge j >0$.\\

When $i>j=0$, $|s(i+2,0;1)|= (i+1)(i+2)|s(i,0;1)|$, when $i=j=0$, $s(2,0;1) = 4 < f(0)s(2,0;1)$. So \ref{appendix_eq_1} still holds.
\end{proof}
\end{appendices}

\section*{Acknowledgments}
I would like to thank Prof. Manfred Denker for his lucid explanations and guidance. 

\bibliographystyle{plain}
\bibliography{bibf}
\end{document}